\documentclass[12pt, reqno]{amsart}

\usepackage{amsmath, amssymb, amsthm}
\usepackage{enumerate}

\begin{document}

 \bibliographystyle{plain}
 \newtheorem{theorem}{Theorem}
 \newtheorem{lemma}[theorem]{Lemma}
 \newtheorem{corollary}[theorem]{Corollary}
 \newtheorem{problem}[theorem]{Problem}
 \newtheorem{conjecture}[theorem]{Conjecture}
 \newtheorem{definition}[theorem]{Definition}
 \newtheorem{prop}[theorem]{Proposition}
 \numberwithin{equation}{section}
 \numberwithin{theorem}{section}

 \newcommand{\mo}{~\mathrm{mod}~}
 \newcommand{\mc}{\mathcal}
 \newcommand{\rar}{\rightarrow}
 \newcommand{\Rar}{\Rightarrow}
 \newcommand{\lar}{\leftarrow}
 \newcommand{\lrar}{\leftrightarrow}
 \newcommand{\Lrar}{\Leftrightarrow}
 \newcommand{\zpz}{\mathbb{Z}/p\mathbb{Z}}
 \newcommand{\mbb}{\mathbb}
 \newcommand{\B}{\mc{B}}
 \newcommand{\cc}{\mc{C}}
 \newcommand{\D}{\mc{D}}
 \newcommand{\E}{\mc{E}}
 \newcommand{\F}{\mathbb{F}}
 \newcommand{\G}{\mc{G}}
  \newcommand{\ZG}{\Z (G)}
 \newcommand{\FN}{\F_n}
 \newcommand{\I}{\mc{I}}
 \newcommand{\J}{\mc{J}}
 \newcommand{\M}{\mc{M}}
 \newcommand{\nn}{\mc{N}}
 \newcommand{\qq}{\mc{Q}}
 \newcommand{\PP}{\mc{P}}
 \newcommand{\U}{\mc{U}}
 \newcommand{\X}{\mc{X}}
 \newcommand{\Y}{\mc{Y}}
 \newcommand{\itQ}{\mc{Q}}
 \newcommand{\sgn}{\mathrm{sgn}}
 \newcommand{\C}{\mathbb{C}}
 \newcommand{\R}{\mathbb{R}}
 \newcommand{\T}{\mathbb{T}}
 \newcommand{\N}{\mathbb{N}}
 \newcommand{\Q}{\mathbb{Q}}
 \newcommand{\Z}{\mathbb{Z}}
 \newcommand{\A}{\mathcal{A}}
 \newcommand{\ff}{\mathfrak F}
 \newcommand{\fb}{f_{\beta}}
 \newcommand{\fg}{f_{\gamma}}
 \newcommand{\gb}{g_{\beta}}
 \newcommand{\vphi}{\varphi}
 \newcommand{\whXq}{\widehat{X}_q(0)}
 \newcommand{\Xnn}{g_{n,N}}
 \newcommand{\lf}{\left\lfloor}
 \newcommand{\rf}{\right\rfloor}
 \newcommand{\lQx}{L_Q(x)}
 \newcommand{\lQQ}{\frac{\lQx}{Q}}
 \newcommand{\rQx}{R_Q(x)}
 \newcommand{\rQQ}{\frac{\rQx}{Q}}
 \newcommand{\elQ}{\ell_Q(\alpha )}
 \newcommand{\oa}{\overline{a}}
 \newcommand{\oI}{\overline{I}}
 \newcommand{\dx}{\text{\rm d}x}
 \newcommand{\dy}{\text{\rm d}y}
\newcommand{\cal}[1]{\mathcal{#1}}
\newcommand{\cH}{{\cal H}}
\newcommand{\diam}{\operatorname{diam}}
\newcommand{\bx}{\mathbf{x}}
\newcommand{\Ps}{\varphi}

\parskip=0.5ex

\title[Quasicrystals and the Littlewood conjecture]{Perfectly ordered quasicrystals\\ and the Littlewood conjecture}
\author[Haynes, Koivusalo, Walton]{Alan~Haynes,~
Henna~Koivusalo,~
James~Walton}
\thanks{Research supported by EPSRC grants EP/L001462, EP/J00149X, EP/M023540}

\allowdisplaybreaks

\begin{abstract}
Linearly repetitive cut and project sets are mathematical models for perfectly ordered quasicrystals. In a previous paper we presented a characterization of linearly repetitive cut and project sets. In this paper we extend the classical definition of linear repetitivity to try to discover whether or not there is a natural class of cut and project sets which are models for quasicrystals which are better than `perfectly ordered'. In the positive direction, we demonstrate an uncountable collection of such sets (in fact, a collection with large Hausdorff dimension) for every choice of dimension of the physical space. On the other hand we show that, for many natural versions of the problems under consideration, the existence of these sets turns out to be equivalent to the negation of a well known open problem in Diophantine approximation, the Littlewood conjecture.
\end{abstract}

\maketitle

\section{Introduction}
\subsection{Statements of results}

A cut and project set $Y\subseteq \R^d$ is {\bf linearly repetitive (LR)}  if there exists a constant $C$ such that, for all sufficiently large $r$, every pattern of diameter $r$, which occurs somewhere in $Y$, occurs in every ball of diameter $Cr$ in $\R^d$. LR cut and project sets were introduced by Lagarias and Pleasants in \cite{LagaPlea2003} as models for `perfectly ordered' quasicrystals. In this paper we focus mostly on what we will refer to as {\bf cubical cut and project sets}, which are regular, totally irrational, aperiodic cut and project sets formed with a {\bf cubical window} (definitions of these terms are provided in the next section). In a previous paper \cite{HaynKoivWalt2015a} we gave a characterization of all LR cubical cut and project sets (as well as some canonical ones). We provided a necessary and sufficient condition that involved an algebraic component, that the sum of the ranks of the kernels of the linear forms defining the cut and project set should be maximal, and a Diophantine component, that the linear forms should be badly approximable when restricted to subspaces complementary to their kernels.

The motivation for this paper is to try to understand whether or not there could exist quasicrystals with even more structure than the perfectly ordered examples described above. In order to begin our discussion we refine the notion of LR as follows. Let $\mc{A}$ be a collection of bounded, measurable subsets of $\R^d$. We say that $Y$ is {\bf LR with respect to $\mathbf{\mc{A}}$} if there exists a constant $C>0$ such that, for every set $\Omega\in\mc{A}$, every pattern of shape $\Omega$ in $Y$ occurs in every ball of volume $C|\Omega|$ in $\R^d$, where $|\Omega|$ denotes the $d$-dimensional Lebesgue measure of $\Omega$. To clarify an important point, when we say that a pattern with a given shape `occurs' in a certain region, we mean that the region contains a point of $Y$ which is the distinguished point of a patch of that shape (precise definitions will be provided in Section \ref{sec.Prelim.Patterns}).

It is easy to see that $Y$ is LR, in the usual sense, if and only if it is LR with respect to a collection $\mc{A}$ consisting of all dilates of a single (and therefore any) fixed bounded convex set with non-empty interior. As an optimistic first question, we may ask whether or not there are cubical cut and project sets, with $d>1$, which are LR with respect to the collection of {\it all} bounded convex sets of volume at least $1$ in $\R^d$. Somewhat more modestly, we might also ask whether or not there are such sets which are LR with respect to the collection of all aligned rectangles of volume at least $1$ in $\R^d$ (we say that a rectangle in $\R^d$ is {\bf aligned} if all of its faces are parallel to coordinate hyperplanes). However, not too surprisingly, the answers to both of these questions turn out to be no, albeit for trivial reasons.

Basic considerations reveal that, in order to make our problem interesting, it is necessary to choose $\mc{A}$ so that there is a constant $\eta>0$ with the property that, for any shape $\Omega\in\mc{A}$, the number of integer points in any translate of $\eta\Omega$ is bounded above by a fixed constant multiple of the volume of $\Omega$. Taking this into consideration, there is more than one logical way to proceed, and for much of the paper we choose to restrict our attention to sets $\mc{A}$ which are collections of polytopes with integer vertices. In Section \ref{sec.AltChoice} we will revisit this decision and discuss another natural choice, collections of convex shapes with inradii uniformly bounded from below.

To begin with, let $\mc{C}_d$ denote the collection of convex polytopes in $\R^d$ with non-empty interior and vertices in $\Z^d$. If $d=1$ then being LR with respect to $\mc{C}_d$ is the same as being LR, in the usual sense. In this case, $k$ to $d$ cubical cut and project sets which are LR exist only when $k=2$. They correspond precisely to lines with badly approximable slopes, and they are therefore naturally parameterized by a collection of real numbers of Hausdorff dimension $1$ (this follows from \cite[Theorem 1.1]{HaynKoivWalt2015a}, but also from results in \cite{BesbBoshLenz2013}). Our first result shows that this is the only case in which such a set can be LR with respect to $\mc{C}_d$.
\begin{theorem}\label{thm.LRConvex}
For any $k$ and $d$ with $(k,d)\not=(2,1)$, there are no $k$ to $d$ cubical cut and project sets which are $\mathrm{LR}$ with respect to $\mc{C}_d.$
\end{theorem}
Next we consider the question of whether or not there are non trivial examples of cut and project sets which are LR with respect to the subset $\mc{R}_d\subseteq\mc{C}_d$ consisting of aligned rectangles with integer vertices. Here the problem turns out to be slightly less straightforward. As our second result shows, answering it in full is equivalent to determining the falsity or truth of a well known long standing open problem, the Littlewood conjecture in Diophantine approximation, and its natural higher dimensional generalizations.
\begin{theorem}\label{thm.LRAlignedBoxes}
Suppose that $k>d\ge 1$. If $k-d>1$ then there are no $k$ to $d$ cubical cut and project sets which are $\mathrm{LR}$ with respect to $\mc{R}_d$. If $k-d=1$ then the following statements are equivalent:
\begin{itemize}
  \item[(C1)] There exists a $k$ to $d$ cubical cut and project set which is $\mathrm{LR}$ with respect to $\mc{R}_d$.\vspace*{.1in}
  \item[(C2)] There exist real numbers $\alpha_{1},\ldots ,\alpha_{d}$ satisfying
      \[\liminf_{n\rar\infty}n\|n\alpha_{1}\|\cdots\|n\alpha_{d}\|>0.\]
      \end{itemize}
\end{theorem}
The proofs of our theorems are based on a collection of observations from tiling theory and Diophantine approximation, which have been developed in several recent works \cite{BertVuil2000,HaynKoivSaduWalt2015,HaynKoivWalt2015a,Juli2010}. In \cite{Juli2010} it was explained how one can translate the problem of studying patterns in cut and project sets to a dual problem of studying connected components of sets in the internal space, defined by a natural (linear) $\Z^k$-action. As shown in \cite{HaynKoivWalt2015a}, the property of linear repetitivity then translates into a question about densities of orbits of points in the internal space under the $\Z^k$-action. With this as a backdrop, the theorems above are manifestations of various Diophantine properties of the subspace $E$ defining $Y$.

For the sake of readers who are not familiar with the Littlewood conjecture we have included a description of it in the next section. The important point is that, for $d>1,$ real numbers $\alpha_1,\ldots , \alpha_d$ satisfying (C2) above, are conjectured not to exist. What we can say definitively is that, from the proof of Theorem \ref{thm.LRAlignedBoxes}, and by a deep theorem by Einsiedler, Katok, and Lindenstrauss \cite[Theorem 1.6]{EinsKatoLind2006}, for $k\ge 3$ the collection of $k$ to $k-1$ cubical cut and project sets which are LR with respect to $\mc{R}_d$, if non-empty, is naturally parameterized by a subset of $\R^d$ with Hausdorff dimension $0$. By way of comparison, it follows from \cite[Corollary 1.3]{HaynKoivWalt2015a} that for $d\ge k/2$, the collection of cubical cut and project sets which are LR, in the usual sense, has Hausdorff dimension $d$.

In the special case of $k=3$ and $d=2$ the theorem above gives an equivalent formulation of the Littlewood conjecture. Furthermore, we have the following immediate corollary.
\begin{corollary}\label{cor.LRifLCTrue}
If the Littlewood conjecture is true then, as long as $(k,d)\not=(2,1)$, there are no $k$ to $d$ cubical cut and project sets which are $\mathrm{LR}$ with respect to $\mc{R}_d$.
\end{corollary}
It seems possible that the connections described above could serve as an indirect route for deriving information about the Littlewood conjecture. On the other hand, from the point of view of discovering very well ordered quasicrystals, the results presented so far leave us with the somewhat unsatisfying impression that, if they exist, such patterns must be exceedingly rare. However, we will now show how a minor adjustment in our generalized definition of LR leads to an abundance of cut and project sets which are indeed more than `perfectly ordered'.

For a collection $\mc{A}$ of bounded subsets of $\R^d$, we say that $Y\subseteq\R^d$ is {\bf $\mathbf{\mathrm{\bf LR}_\Omega}$ with respect to $\mathbf{\mc{A}}$} if there is a constant $C>0$ such that, for every set $\Omega\in\mc{A}$, every pattern of shape $\Omega$ in $Y$ occurs in every translate of $C\Omega$ in $\R^d$. The only difference between LR and $\mathrm{LR}_\Omega$ is that, {\it in the definition of $\mathrm{LR}_\Omega$, we search for patterns of a given shape in a region which is a dilate of the same shape.} As before, when $\mc{A}$ consists of all dilations of a fixed bounded convex set, the definition of `$\mathrm{LR}_\Omega$ with respect to $\mc{A}$' reduces to the original definition of LR.

First of all, for much the same reason as Theorem \ref{thm.LRConvex}, we have the following result.
\begin{theorem}\label{thm.LR'Convex}
For any $k$ and $d$ with $(k,d)\not=(2,1)$, there are no $k$ to $d$ cubical cut and project sets which are $\mathrm{LR}_\Omega$ with respect to $\mc{C}_d$.
\end{theorem}
Perhaps more surprisingly, in contrast with Theorem \ref{thm.LRAlignedBoxes}, we obtain the existence of uncountably many `super perfectly ordered' quasicrystals, when $\mc{C}_d$ is replaced by $\mc{R}_d$.
\begin{theorem}\label{thm.LR'AlignedBoxes}
For any $d\ge 1$ the set of $2d$ to $d$ cubical cut and project sets which are $\mathrm{LR}_\Omega$ with respect to $\mc{R}_d$, has Hausdorff dimension equal to $d$.
\end{theorem}
Our proof of this theorem also leads to an explicit method, described in \cite[Section 6]{HaynKoivWalt2015a}, for constructing such sets.

For $k$ to $d$ sets with $k\not=2d$, the situation is different from above. It turns out that for $d< k<2d$, $k$ to $d$ cubical cut and project sets which are $\mathrm{LR}_\Omega$ with respect to $\mc{R}_d$ are less likely to exist. This is demonstrated by the following theorem.
\begin{theorem}\label{thm.LR',k->d}
For any $k>d\ge 1$ the following are equivalent:
\begin{itemize}
  \item[(C1')] There exists a $k$ to $d$ cubical cut and project set which is $\mathrm{LR}_\Omega$ with respect to $\mc{R}_d$.\vspace*{.1in}
  \item[(C2')] There exist positive integers $m_1,\ldots , m_{k-d}$ with $d=m_1+\cdots +m_{k-d}$ and such that, for each $1\le i\le k-d$, we can find $\alpha_{i1},\ldots ,\alpha_{im_i}$ satisfying
      \[\liminf_{n\rar\infty}n\|n\alpha_{i1}\|\cdots\|n\alpha_{im_i}\|>0.\]
\end{itemize}
\end{theorem}
For $d>1$ the second condition in this theorem is predicated on the existence of counterexamples to the Littlewood Conjecture. Therefore, in analogy with Corollary \ref{cor.LRifLCTrue}, we obtain the following result.
\begin{corollary}\label{cor.LR'ifLCTrue}
If the Littlewood conjecture is true then, for any $d\ge 1$, and for any $k\not=2d$, there are no $k$ to $d$ cubical cut and project sets which are $\mathrm{LR}_\Omega$ with respect to $\mc{R}_d$.
\end{corollary}
Finally, we mention that many interesting cut and project sets can be obtained by using a {\bf canonical window} instead of a cubical one (see the definitions in the next section). We refer to such sets as {\bf canonical cut and project sets}. There are subtle technical issues in trying to move from a cubical cut and project set to the canonical cut and project set obtained from the same data (i.e. with the cubical window replaced by a canonical one). Perhaps surprisingly, there are examples in which the cubical cut and project set and its canonical counterpart are not mutually locally derivable (see \cite{Sadu2008} for the precise definition of this term). On the other hand, for many specific examples of interest (e.g. physical and internal space pairs which give rise to Amman-Beenker and Penrose tilings), we are able to argue directly to understand the structure of canonical cut and project sets from the corresponding cubical ones. These issues are discussed in some detail in \cite{HaynKoivWalt2015a}.

For the purposes of this paper, our results about cubical cut and project sets which are LR or $\mathrm{LR}_\Omega$ with respect to $\mc{C}_d$ do not immediately extend to canonical ones. The difficulty is essentially due to the fact that there are convex polytopes with integer vertices and arbitrarily small inradius. However, some of our results for $\mc{R}_d$ can be made to apply to canonical cut and project sets as well.
\begin{theorem}\label{thm.Canonical}
  Theorem \ref{thm.LRAlignedBoxes}, as well as Corollaries \ref{cor.LRifLCTrue} and \ref{cor.LR'ifLCTrue}, are true with the adjective`cubical' replaced by `canonical'.
\end{theorem}

This paper is organized as follows: In Section \ref{sec.Prelim} we will give details and definitions of our objects of study, and we will explain relevant results from previous work, laying the groundwork for proofs in subsequent sections. In Sections \ref{sec.LRThm1}-\ref{sec.LR'Thms2,3} we will present the proofs of our results about cubical cut and project sets. In Section \ref{sec.Canonical} we will present the proof of Theorem \ref{thm.Canonical}. In Section \ref{sec.AltChoice} we will discuss a natural alternate choice of shapes which can be considered in place of $\mc{C}_d$, the collection $\mc{C}_d'$ of convex sets with inradius at least $1/2$. The proofs of our results about $\mc{C}_d$ do not extend immediately to $\mc{C}_d'$, and this raises an interesting open problem which has strong connections to Diophantine approximation.

\subsection{Notation}\label{sec.Notation}
For sets $A$ and $B$, the notation $A\times B$ denotes the Cartesian product. If $A$ and $B$ are subsets of the same Abelian group, then $A+B$ denotes the collection of all elements of the form $a+b$ with $a\in A$ and $b\in B$.

For $x\in\R,~\{x\}$ denotes the fractional part of $x$ and
$\| x \|$ denotes the distance from $x$ to the nearest integer. For
$x\in\R^m$, we set
$|x|=\max\{|x_1|,\ldots ,|x_m|\}$ and
$\|x\|=\max\{\|x_1\|,\ldots ,\|x_m\|\}.$ We use the symbols $\ll, \gg,$ and $\asymp$ for the standard Vinogradov and asymptotic notation.

\section{Preliminary results}\label{sec.Prelim}

\subsection{Cut and project sets}\label{sec.CutAndProj}
For the most part, we are using the same setup as in \cite{HaynKoivWalt2015a}. However, for completeness and to avoid confusion, we provide all of our definitions here. Let $E$ be a $d$-dimensional subspace of $\R^k$, and $F_\pi\subseteq\R^k$ a subspace complementary to $E$. Write $\pi$ for the projection onto $E$ with respect to the decomposition $\R^k=E+F_\pi$. Choose a set $\mc{W}_\pi\subseteq F_\pi$, and define $\mc{S}=\mc{W}_\pi+E$. The set $\mc{W}_\pi$ is referred to as the {\bf window}, and $\mc{S}$ as the {\bf strip}. For each $s\in\R^k/\Z^k,$ we define the {\bf cut and project set} $Y_s\subseteq E$ by
\[Y_s=\pi(\mc{S}\cap(\Z^k+s)).\]
In this situation we refer to $Y_s$ as a {\bf $\mathbf{k}$ to $\mathbf{d}$ cut and project set}.

We adopt the conventional assumption that $\pi|_{\Z^k}$ is injective. We also assume in much of what follows that $E$ is a {\bf totally irrational} subspace of $\R^k$, which means that the canonical projection of $E$ into $\R^k/\Z^k$ is dense. There is little loss of generality in this assumption, since any subspace of $\R^k$ is dense in some rational sub-torus of $\R^k/\Z^k$.

For the problem of studying LR, the $s$ in the definition of $Y_s$ plays only a minor role. If we restrict our attention to points $s$ for which $\Z^k+s$ does not intersect the boundary of $\mc{S}$ (these are called {\bf regular} points) then, as long as $E$ is totally irrational, the sets of finite patches in $Y_s$ do not depend on the choice of $s$. In particular, the property of being LR with respect to some collection of sets does not depend on the choice of $s$, as long as $s$ is taken to be a regular point. On the other hand, for points $s$ which are not regular, the cut and project set $Y_s$ may contain `additional' patches coming from points on the boundary, which will make it non-repetitive, and therefore not LR, but for superficial reasons. For this reason, {\it we will always assume that $s$ is taken to be a regular point, and we will often simplify our notation by writing $Y$ instead of $Y_s$}.

As a point of reference, when allowing $E$ to vary, we also make use of the fixed subspace $F_\rho=\{0\} \times \R^{k-d}\subseteq \R^k$, and we define $\rho:\R^k\rar E$ and $\rho^*:\R^k\rar F_\rho$ to be the projections onto $E$ and $F_\rho$ with respect to the decomposition $\R^k=E+F_\rho$ (recall that we are assuming $E$ is totally irrational). Our notational use of $\pi$ and $\rho$ is intended to be suggestive of the fact that $F_\pi$ is the subspace which gives the {\em projection} defining $Y$ (hence the letter $\pi$), while $F_\rho$ is the subspace with which we {\em reference} $E$ (hence the letter $\rho$). We write $\mc{W}=\mc{S}\cap F_\rho$, and for convenience we also refer to this set as the {\bf window} defining $Y$. This slight ambiguity should not cause any confusion in the arguments below.

For some problems about cut and project sets (e.g. the deformation properties considered in \cite{HaynKellWeis2014}) we are able to present interesting results with very weak assumption on the window $\mc{W}$. However, for problems about regularity of patterns, small pathologies in the window lead to sparse but erratic behavior in the corresponding cut and project sets. The property of being LR is quite restrictive and, in order for it to hold, it is necessary that the window be compatible with the lattice $\Z^k$ in some way. Therefore in much of this paper we will focus our attention on the situation where $\mc{W}$ is taken to be a {\bf cubical window}, given by
\begin{equation}\label{eqn.SquareWindow}
\mc{W}=\left\{\sum_{i=d+1}^{k}t_ie_i:0\le t_i<1\right\}.
\end{equation}
In Section \ref{sec.Canonical} we will also consider the case when $\mc{W}$ is taken to be a {\bf canonical window}, i.e. the $(k-d)$-dimensional polytope which is the image under $\rho^*$ of the unit cube in $\R^k$.

For any cut and project set, the collection of points $x\in E$ with the property that $Y+x=Y$ forms a group, called the {\bf group of periods} of $Y$. We say that $Y$ is {\bf aperiodic} if the group of periods is $\{0\}$. Finally, as mentioned in the introduction, we say that $Y$ is a {\bf cubical} (resp. {\bf canonical}) {\bf cut and project set} if it is regular, totally irrational, and aperiodic, and if $\mc{W}$ is a cubical (resp. canonical) window.

If $E$ is totally irrational, we can write it as the graph of a linear function with respect to the standard basis vectors in $F_\rho$. In other words,
\begin{equation*}
E = \{(x,L(x)): x \in \R^d\},
\end{equation*}
where $L: \R^d \to \R^{k-d}$ is a linear function. For each $1\le i\le k-d$, we define the linear form $L_i:\R^d \to \R$ by
\[L_i(x) = L(x)_{i} =\sum_{j=1}^d \alpha_{ij} x_j,\]
and we use the points
$\{\alpha_{ij}\}\in\R^{d(k-d)}$ to parametrize the choice of $E$.

\subsection{Diophantine approximation and transference}\label{sec.Approx} Dirichlet's Theorem in Diophantine approximation says that, for any real number $\alpha$, and for any $N\in\N$,
\[\min_{1\le n\le N}\|n\alpha\|\le (N+1)^{-1}.\]
An immediate corollary of this is that
\[\liminf_{n\rar\infty}n\|n\alpha\|\le 1.\]
It follows from a theorem of Borel and Bernstein (or Khintchine's Theorem, which gives a stronger result) that, for Lebesgue almost every $\alpha$,
\[\liminf_{n\rar\infty}n\|n\alpha\|=0.\]
On the other hand, it is a theorem of Jarnik that the set of $\alpha$ for which
\[\liminf_{n\rar\infty} n\|n\alpha\|>0,\]
is a set of Hausdorff dimension $1$.

The Littlewood conjecture, proposed by J.~E.~Littlewood, is the conjecture that, for every pair of real numbers $\alpha$ and $\beta$, we have that
\[\liminf_{n\rar\infty}n\|n\alpha\|\|n\beta\|=0.\]
Important advances in the understanding of the Littlewood conjecture have been made by several authors, including Cassels and Swinnerton-Dyer \cite{CassSwin1955}, Pollington and Velani \cite{PollVela2000}, and Badziahin, Pollington, and Velani \cite{BadzPollVela2011}. The metric (a.e.) theory of this problem is well understood, thanks largely to the work of Gallagher \cite{Gall1962} (see also \cite{BereHaynVela2015}), and it is also known, due to results of Einsiedler, Katok, and Lindenstrauss \cite{EinsKatoLind2006},  that the set of $(\alpha,\beta)\in\R^2$ which do not satisfy the Littlewood conjecture is a set of Hausdorff dimension $0$. However the original conjecture remains an open problem.

For $m\ge 2$, we will call the $m$-dimensional Littlewood conjecture the assertion that, for any $\alpha_1,\ldots ,\alpha_m\in\R$
\[\liminf_{n\rar\infty}n\|n\alpha_1\|\cdots \|n\alpha_m\|=0.\]
Analogues of most of the above mentioned results exist for $m>2$, although the boundary of what is known is not significantly different for larger $m$ than it is for the $m=2$ problem.

In the proofs of our main results we will use the following `dual' form of the above problems.
\begin{lemma}\label{lem.LCDualForm}
Suppose that $m\ge 1$. The number $(\alpha_1,\ldots ,\alpha_m)\in\R^m$ satisfies
\[\liminf_{n\rar\infty}n\|n\alpha_1\|\cdots \|n\alpha_m\|=\epsilon,\]
for some $\epsilon>0$, if and only if there exists a constant $c>0$ such that, for all nonzero integers $n\in\Z^m$,
\[\|n_1\alpha_1+\cdots +n_m\alpha_m\|> \frac{c}{(1+|n_1|)\cdots (1+|n_m|)}.\]
Furthermore the constant $c$ can be made to depend only on $\epsilon$, and not on $(\alpha_1,\ldots ,\alpha_m)$.
\end{lemma}
\begin{proof}
For $m=1$ this is obvious. For $m\ge 2$ it follows directly from the results of Mahler in \cite{Mahl1939}. See also \cite[Appendix]{BadzPollVela2011} and \cite[Lemma 1]{Bere2014}.
\end{proof}
We will also use a transference principle which allows us to go from a potential counterexample to the $m$-dimensional Littlewood conjecture, to a corresponding inhomogeneous problem for aligned boxes.
\begin{lemma}\label{lem.LCInhomogForm}
For $m\ge 2$, if $(\alpha_1,\ldots,\alpha_m)$  is a counterexample to the $m$-dimensional Littlewood conjecture then there is a constant $C>0$, with the property that, for any $N_1,\ldots ,N_m\in\N$, the collection of points
\[\left\{\{n_1\alpha_1+\cdots +n_m\alpha_m\}:|n_i|\le N_i\right\}\]
is $C/(N_1\cdots N_m)$-dense in $\R/\Z$. If $m=1$ and $\alpha_1$ is a badly approximable number then this statement is also true.
\end{lemma}
\begin{proof}
For $m=1$ this is precisely \cite[Section V, Theorem VI]{Cass1957}, and for $m\ge 2$  it is a modification of the proof of that theorem. For completeness we provide the details of the argument.

If $(\alpha_1,\ldots ,\alpha_m)$ is a counterexample to the $m$-dimensional Littlewood conjecture then by Lemma \ref{lem.LCDualForm} there is a constant $c>0$ such that, for any $N_1,\ldots ,N_m\in\N$, and for any nonzero $n\in\Z^m$ with $|n_i|\le N_i$ for all $i$, we have that
\[\|n_1\alpha_1+\cdots +n_m\alpha_m\|>\frac{c}{N_1\cdots N_m}.\]
For $1\le i\le m+1,$ define linear forms $f_i:\R^{m+1}\rar\R$ by
\begin{align*}
  f_1(x)&=(N_1\cdots N_m/c)\cdot (x_1\alpha_1+\cdots +x_m\alpha_m+x_{m+1}),\\
  f_2(x)&=x_1/N_1, ~f_3(x)=x_2/N_2,~\ldots ~,~ f_{m+1}(x)=x_m/N_m.
\end{align*}
The matrix defining these forms has determinant $\pm1/c$, and there is no nonzero $n\in\Z^{m+1}$ for which
\[\max_i|f_i(n)|<1.\]
Therefore, by \cite[Section V, Theorem V]{Cass1957}, for every $\gamma\in\R^{m+1}$, there is an integer $n\in\Z^{m+1}$ for which
\[\max_i|f_i(n)-\gamma_i|<\frac{1}{2}\left(\frac{1}{c}+1\right).\]
It is clear from this that we can choose $C$ so that it satisfies the claim in the statement of the lemma.
\end{proof}

\subsection{Patterns and regular points}\label{sec.Prelim.Patterns}
For $y\in Y_s$ we will use the notation $\tilde{y}$ to denote the point in $\Z^k$ which satisfies $\pi (\tilde{y}+s)=y$. Since $\pi|_{\Z^k}$ is injective, this point is uniquely defined.

In our discussion in the introduction we referred to the shapes in the collection $\mc{A}$, as well as the regions in which we search for them in our two notions of repetitivity, as subsets of $\R^d$. It is necessary to be more precise, since we are actually working in $\R^k$, so we will make the convention that these sets are taken to be subsets of $F_\rho^\perp=\langle e_1,\ldots ,e_d\rangle_\R$. The definitions of $\mc{C}_d$ and $\mc{R}_d$ can then be read exactly as before. From the point of view of working within $E$, all of these sets can be thought of as the corresponding images under the map $\rho$.

For each $\Omega\in\mc{A}$ and for each $y\in Y$, we
define the {\bf patch of shape $\mathbf{\Omega}$ at $\mathbf{y}$}, by
\[P(y,\Omega):=\{ y' \in Y: \rho (\tilde{y'}-\tilde y) \in \rho(\Omega)\}.\]
In other words, $P(y,\Omega)$ consists of the projections (under $\pi$) to $Y$ of all points of $\mc{S}$ whose first $d$ coordinates are in a certain neighborhood, determined by $\Omega$ and the first $d$ coordinates of $\tilde y$. The reader may wish to see the discussion in \cite[Section 2.3]{HaynKoivWalt2015a} of how this relates to other existing notions in the literature of patterns in cut and project sets.

For $y_1,y_2\in Y$, we say that $P(y_1,\Omega)$ and
$P(y_2,\Omega)$ are equivalent if
\[P(y_1,\Omega)=P(y_2,\Omega)+y_1-y_2.\]
This defines an equivalence relation on the collection of patches of shape $\Omega$. We denote the equivalence class of the patch of shape $\Omega$ at $y$ by $\mc{P}(y,\Omega)$. Note that it is possible for two patches which are translates of each other, as point sets, to fall in different equivalence classes. This highlights the importance of the role of $y$, the {\bf distinguished point}, in the definition of $P(y,\Omega)$.

There is a natural action of $\Z^k$ on $F_\rho$, given by
\[n.w=\rho^*(n)+w = w + (0,n_2-L(n_1)),\]
for $n=(n_1,n_2)\in\Z^k = \Z^d \times
\Z^{k-d}$ and $w\in F_\rho$. For each $\Omega\in\mc{A}$ we define the
{\bf $\mathbf{\Omega}$-singular points} of $\mc{W}$ by
\[\mathrm{sing}(\Omega):=\mc{W}\cap\left((-(\rho^{-1}\circ\rho)(\Omega)\cap\Z^k). \partial\mc{W}\right),\]
and the {\bf $\mathbf{\Omega}$-regular points} by
\[\mathrm{reg}(\Omega):=\mc{W}\setminus\mathrm{sing}(\Omega).\]
The singular points are just the translates of the boundary of $\Omega$ under the natural action of the (negatives of) the collection of integer points in $\Z^k$ whose first $d$ coordinates lie in $\Omega$. The following result follows from the proof of \cite[Lemma 3.2]{HaynKoivSaduWalt2015} (see also \cite{Juli2010}).
\begin{lemma}\label{lem.ConnComp}
Suppose that $\mc{W}$ is a parallelotope
generated by integer vectors, and suppose that $\Omega\in\mc{A}$ is a convex set with non-empty interior. For every equivalence class $\mc{P}=\mc{P}(y,\Omega)$, there is a unique connected component $U$ of $\mathrm{reg}(\Omega)$ with the property that, for any $y'\in Y_s$,
\[\mc{P}(y',\Omega)=\mc{P}(y,\Omega)~\text{ if and only if }~ \rho^*(\tilde{y'}+s)\in U.\]
\end{lemma}
This lemma is an important tool which will allow us to translate problems about patterns in $Y$ into the language of Diophantine approximation in $F_\rho$.

\section{Proof of Theorem \ref{thm.LRConvex}}\label{sec.LRThm1}
When $k-d>1,$ the result of Theorem \ref{thm.LRConvex} follows from Theorem \ref{thm.LRAlignedBoxes}, which will be proved in the next section. Therefore we will assume the validity of the second theorem (proved in the next section), and suppose that $k-d=1$ and that $d>1$ (for the $d=1$ case see the comments immediately preceding the statement of the theorem). In this case, the subspace $E$ is the graph of a single linear form in $d$ variables, which we write as
\[L(x)=\sum_{j=1}^d\alpha_jx_j.\]

Let $B(x,r)$ denote the sup-norm ball centered at $x\in\R^d$, of radius $r>0$. By basic geometric considerations (see \cite[Equation (4.1)]{HaynKoivSaduWalt2015}) there is a constant $c>0$ with the property that, for any $r>0$ and for any $y\in Y$, the collection of points $y'\in Y$ satisfying
\[y'-y\in\rho (B(0,r))\]
is a subset of the patch
\[P(y,B(0,r+c)).\]

For each $N\in\N$ and for each matrix $A\in\mathrm{SL}_d(\Z)$ let $\Omega_{A,N}\in\mc{C}_d$ be defined by
\[\Omega_{A,N}=A\cdot [-N,N]^d.\]
It follows from our comments in the previous paragraph that there is an $\eta>0$ with the property that, for any $C\ge 1$ and $y\in Y$, the collection of points $y'\in Y$ with \[y'-y\in\rho(B(0,(C|\Omega_{A,N}|)^{1/d}))\]
is a subset of
\[P(y,B(0,(\eta C|\Omega_{A,N}|)^{1/d})).\]
This region depends on $N$ but not on $A$ and, by Lemma \ref{lem.ConnComp}, the collection of patterns of shape $\Omega_{A,N}$ which we see in the region is determined precisely by the collection of connected components of $\mathrm{reg}(\Omega_{A,N})$ which intersect the set
\[O_N(y)=\{\rho^*(\tilde{y}+n+s):\tilde{y}+n+s\in\mc{S},|(n_1,\ldots ,n_d)|\le (\eta C |\Omega_{A,N}|)^{1/d} \}.\]
To elucidate this further, note that for each choice of $(n_1,\ldots ,n_d)\in\Z^d$, there is precisely one point $(n_{d+1},\ldots ,n_k)\in\Z^{k-d}$ with the property that $\tilde{y}+(n_1,\ldots ,n_k)+s\in\mc{S}$. The set $O_N(y)$ therefore represents the orbit in $\mc{W}$ (i.e. modulo $1$) of the initial point $y^*=\rho^*(\tilde{y}+s)$, under the action of the collection of points $n\in\Z^k$ with $|(n_1,\ldots ,n_d)|\le (\eta C |\Omega_{A,N}|)^{1/d}$.

By total irrationality, the collection of points $y^*$, for $y\in Y$, is dense in $\mc{W}$. Therefore, to show that $Y$ is not LR with respect to $\mc{C}_d$, it is sufficient to show that, for any $C\ge 1$, we can choose $A$ and $N$ as above so that there is some regular point in $\mc{W}$ whose orbit under the collection of integers mentioned in the previous paragraph does not intersect one of the connected components of $\mathrm{reg}(\Omega_{A,N})$.

The number of integer points in the orbit we are considering is bounded above by a constant multiple of $N^d$, where the constant depends on $C$ and $\eta$ but nothing else. Therefore we can always choose a component interval {\it of the orbit} which has length $>C'/N^d$, for some $C'>0$ depending on $C$ and $\eta$. Furthermore, as already remarked, we can choose $y\in Y$ to position the left endpoint of this component interval as close to any point in $\mc{W}$ as we like.

On the other hand we will show that, for fixed $N$, we can choose $A$ so that there is a connected component of $\mathrm{reg}(\Omega_{A,N})$ which is as small as we like. We have that
\begin{align*}
\mathrm{sing}(\Omega_{A,N})&=\left\{\{L(n)\}:n\in\Omega_{A,N}\cap\Z^d\right\}\\
&=\left\{\{(\alpha_1,\ldots ,\alpha_d)A\cdot n\} : n\in\Z^d, |n|\le N\right\}.
\end{align*}
Write $A=(a_{ij})$ and set
\[(\beta_1,\ldots ,\beta_d)=(\alpha_1,\ldots ,\alpha_d)A.\]
We claim that, as $A$ runs over $\mathrm{SL}_d(\Z)$, the values of $\beta_1$ are dense modulo $1$. To see why this is true, first notice that the aperiodicity of $Y$ implies that the numbers $1,\alpha_1,\ldots , \alpha_d$ are $\Q$-linearly independent. Therefore the collection of numbers
\[\left\{\sum_{i=1}^d\alpha_i a_i~:~ a\in\Z^d,~ \mathrm{gcd}(a_1,\ldots ,a_d)=1\right\}\]
is dense modulo $1$. The density of the values of $\{\beta_1\}$ then follows from the fact that any vector $a\in\Z^d$ with $\mathrm{gcd}(a_1,\ldots ,a_d)=1$ may be extended to a basis of $\Z^d$ (see \cite[Chapter 1, Section 3, Theorem 5]{GrubLekk1987}).

The points $0$ and $\beta_1$ are always elements of $\mathrm{sing}(\Omega_{A,N})$. Since we can choose $A$ to make $\beta_1$ as close to $0$ as we like, we can ensure that there is a component interval of $\mathrm{reg}(\Omega_{A,N})$ which has length $<C'/N^d$. These observations together complete the proof that $Y$ is not LR with respect to $\mc{C}_d'$.

\section{Proof of Theorem \ref{thm.LRAlignedBoxes}}\label{sec.LRThm2}
For the proof of Theorem \ref{thm.LRAlignedBoxes} we will need to use the machinery developed in our classification of LR cut and project sets, i.e. the proof of \cite[Theorem 1.1]{HaynKoivWalt2015a}. Following the notation in Section \ref{sec.CutAndProj}, suppose that $Y$ is a $k$ to $d$ cubical cut and project set defined by linear forms $\{L_i\}_{i=1}^{k-d}$.

Assume first that $k-d>1$. For each $1\le i\le k-d$ define a map  $\mc{L}_i:\Z^d\rar\R/\Z$ by
\[\mc{L}_i(n)=L_i(n)~\mathrm{mod}~1,\]
and let $S_i\leqslant\Z^d$ denote the kernel of $\mc{L}_i$. Then for each $i$ define $\Lambda_i\leqslant\Z^d$ by
\[\Lambda_i=\bigcap_{\substack{j=1\\j\not= i}}^{k-d}S_j,\]
and let $\Lambda=\Lambda_1+\cdots +\Lambda_{k-d}$.

It is not difficult to check that if $Y$ is LR with respect to $\mc{R}_d$ then it is LR in the original sense (this follows almost immediately from the definitions). One of the crucial points in the proof of \cite[Theorem 1.1]{HaynKoivWalt2015a} established that if $Y$ is LR, then
\[\mathrm{rk}(\Lambda_i+S_i)=d\quad\text{for each}\quad 1\le i\le k-d,\]
and, in addition, that
\[\mathrm{rk}(\Lambda)=d.\]
Viewed another way, this means that if $Y$ is LR then we can find a basis for a sublattice of $\Z^d$, of full rank, with respect to which the matrix $(\alpha_{ij})$ defined by the linear forms $L_i$ becomes block diagonal.

Let $m_i$ denote the rank of $\Lambda_i$ and note that, by total irrationality, $m_i\ge 1$. We will now show that, if $Y$ is LR with respect to $\mc{R}_d$ then, for each $i$, the real subspace $X_i$ generated by $\Lambda_i$ is actually an $m_i$-dimensional coordinate plane, i.e. a subspace generated by $m_i$ of the standard basis vectors $e_1,\ldots ,e_d$.

Since $[\Z^d:\Lambda]<\infty,$ for each $1\le j\le d$ we can choose a positive integer $n_j$ so that
\begin{equation}\label{eqn.n_je_jExpansion}
n_je_j=\sum_{i=1}^{k-d}\lambda_{ij},
\end{equation}
with $\lambda_{ij}\in \Lambda_i$ for each $i$. Then, for each $N\in\N$ we define $\Omega_N^{(j)}\in\mc{R}_d$ by
\[\Omega_N^{(j)}=\{x\in\R^d : |x_j|\le n_jN~\text{ and }~ |x_i|\le 1 \text{ for }i\not=j\}.\]
For each $1\le i\le k-d,$ as $n$ runs over all elements of $\Omega_N^{(j)}\cap\Z^d$, the number of distinct values taken by $\mc{L}_i(n)$ is bounded above by $3^{d-1}$ if $\lambda_{ij}=0$, otherwise it is at least $2N+1$. This implies that the number of connected components of $\mathrm{reg}(\Omega_N^{(j)})$ is $\gg N^{\kappa_j}$, where $\kappa_j$ is the number of non-zero summands on the right hand side of \eqref{eqn.n_je_jExpansion}. For any constant $C>0$, the number of integer points in a ball of volume $C|\Omega_N^{(j)}|$ is $\ll CN$, so if $\kappa_j>1$ then for $N$ large enough it is impossible for such a ball to contain every patch of shape $\Omega_N^{(j)}$. This shows that if $Y$ is LR with respect to $\mc{R}_d$ then each of the standard basis vectors is contained in one of the subspaces $X_i$. By rank considerations, it follows that each of the subspaces $X_i$ is an $m_i$-dimensional coordinate plane.

Without loss of generality, by relabeling if necessary, assume that
\[X_1=\langle e_1,\ldots ,e_{m_1}\rangle_\R.\]
As $n$ runs over the elements of $\Omega_N^{(1)}\cap\Z^d$, the number of distinct values taken by $\mc{L}_1(n)$ is $\gg N$. However, for any $C>0$, as $n$ runs over the elements of $\Z^d$ in a ball of volume $C|\Omega_N^{(1)}|$, the number of distinct values taken by $\mc{L}_1(n)$ is $\ll (CN)^{m_1/d}$. Since $k-d>1$ and $m_1+\cdots +m_{k-d}=d$, we have that $N^{m_1/d}=o(N)$. This means that for $N$ large enough, it is impossible for the orbits of points in $\mc{W}$, under the action of the integers in a ball of volume $C|\Omega_N^{(1)}|$, to intersect every connected component of $\mathrm{reg}(\Omega_N^{(1)})$. Therefore, by the same argument as used in the previous section, the set $Y$ cannot be LR with respect to $\mc{R}_d$. This completes the proof of the $k-d>1$ case of Theorem \ref{thm.LRAlignedBoxes}.

Next suppose that $k-d=1$ and that $\alpha_1,\ldots ,\alpha_d\in\R$ satisfy condition (C2) in the statement of Theorem \ref{thm.LRAlignedBoxes}.
Let $E\subseteq\R^k$ be the subspace defined by the linear form
\begin{equation}\label{eqn.DefofL}
L(x)=\sum_{i=1}^dx_i\alpha_i,
\end{equation}
and let $Y$ be a cubical cut and project set obtained from $E$. By Lemma \ref{lem.LCDualForm} the numbers $1,\alpha_1,\ldots ,\alpha_d$ are $\Q$-linearly independent, from which it follows that $E$ is totally irrational, and that $Y$ is aperiodic.

From Lemma \ref{lem.LCDualForm} we also have that there is a constant $c>0$ with the property that, for any $\Omega\in\mc{R}_d$, and for any component interval $I$ of $\mathrm{reg}(\Omega)$,
\[|I|>\frac{c}{|\Omega|}.\]
On the other hand, by Lemma \ref{lem.LCInhomogForm} we can choose a constant $C>0$ so that, for any $\Omega\in\mc{R}_d$, the orbit of any regular point in $\mc{W}$ under the collection of integers in a ball of volume $C|\Omega|$ is $c/|\Omega|$-dense in $\mc{W}$. By the argument given in the previous section, this proves that $Y$ is LR with respect to $\mc{R}_d$.

Finally, suppose that $k-d=1,~d>1$, that $E$ is a totally irrational subspace defined by a linear form $L$ as in \eqref{eqn.DefofL}, with real numbers $\alpha_1,\ldots ,\alpha_d$ which do not satisfy condition (C2), and that $Y$ is a cubical cut and project set defined using this data. By Lemma \ref{lem.LCDualForm}, for every $c>0$ we can find an integer $n\in\Z^d$ with
\[\|n_1\alpha_1+\cdots +n_d\alpha_d\|<\frac{c}{(1+|n_1|)\cdots (1+|n_d|)}.\]
This implies that, for any $c>0$, we can find a shape $\Omega\in\mc{R}_d$ and a component interval $I$ of $\mathrm{reg}(\Omega)$ with
\[|I|<\frac{c}{|\Omega|}.\]
On the other hand, there is a constant $\eta>0$ with the property that, for any $C>0$, the number of integer points in a ball of volume $C|\Omega|$ is bounded above by $\eta C|\Omega|.$ Therefore, by the same argument used in the proof in the previous section, for any $C>0$ we can always choose $c>0$ small enough, and a corresponding shape $\Omega$ as above, so that there is a point in $\mc{W}$ whose orbit under the collection of integers in a ball of volume $C|\Omega|$ does not intersect every component interval of $\mathrm{reg}(\Omega)$. This proves that $Y$ is not LR with respect to $\mc{R}_d$, completing the proof of Theorem \ref{thm.LRAlignedBoxes}.

\section{Proof of Theorem \ref{thm.LR'Convex}}\label{sec.LR'Thm1}
The $k-d=1$ cases of Theorem \ref{thm.LR'Convex} follow from the same argument used in the proof of Theorem \ref{thm.LRConvex} above. Note that in the end of that proof we only needed to know that the {\it number} of points in a ball of volume $C|\Omega_{A,N}|$ is $\ll CN^d$. If the ball is replaced by the shape $C\Omega_{A,N}$ then this number is still $\ll C^dN^d,$ and the rest of the proof works as before. The conclusion is that $Y$ cannot be $\mathrm{LR}_\Omega$ with respect to $\mc{C}_d$, unless $k=2$ and $d=1$.

For the $k-d>1$ case of Theorem \ref{thm.LR'Convex} we will use some of the ideas from the beginning of the proof of Theorem \ref{thm.LRAlignedBoxes}. If $Y$ is a cubical cut and project set which is $\mathrm{LR}_\Omega$ with respect to $\mc{C}_d$ then, by just considering the subset of squares in $\mc{C}_d,$ it follows that $Y$ is LR in the usual sense. Therefore, the comments at the beginning of the proof of Theorem \ref{thm.LRAlignedBoxes} apply. Using the notation there, for each $1\le i\le k-d$ choose a non-zero element $\lambda_i\in\Lambda_i$, and then set
\[v=\lambda_1+\cdots +\lambda_{k-d}.\]
For each $N\in\N$ let $\Omega_N\in\mc{C}_d$ be the convex hull of the collection of points
\[\{e_1,\ldots ,e_d\}\cup\{e_i+Nv: 1\le i\le d\}.\]
For each $i$, as $n$ runs over $\Omega_N\cap \Z^d$, the number of distinct values taken by $\mc{L}_i(n)$ is $\gg N$. It follows that the number of connected components of $\mathrm{reg}(\Omega_N)$ is $\gg N^{k-d}$. However, for any $C>0$, the number of points in $C\Omega_N\cap\Z^d$ is $\ll C^d N$. It is clear from this that for large enough $N$, orbits of regular points in $\mc{W}$ under the action of the integers in $C\Omega_N$ cannot intersect every connected component of $\mathrm{reg}(\Omega_N)$. This contradicts our original assumption, forcing us to conclude that $Y$ cannot be $\mathrm{LR}_\Omega$ with respect to $\mc{C}_d$.

\section{Proofs of Theorems \ref{thm.LR'AlignedBoxes} and \ref{thm.LR',k->d}}\label{sec.LR'Thms2,3}
First we present the proof of Theorem \ref{thm.LR',k->d}. The statement of Theorem \ref{thm.LR'AlignedBoxes} will follow easily from our proof.

For one direction of the proof, suppose that (C2') is satisfied and let $Y$ be a $k$ to $d$ cubical cut and project set defined using linear forms
\begin{equation}\label{eqn.BlockDiagLinForms}
L_i(x)=\sum_{j=1}^{m_i}x_{M_i+j}\alpha_{ij},~1\le i\le k-d,
\end{equation}
where $M_1=0$ and $M_i=m_1+\cdots +m_{i-1}$ for $i\ge 2$. As before, the facts that $E$ is totally irrational and that $Y$ is aperiodic follow from Lemma \ref{lem.LCDualForm}.

Suppose that $N_1,\ldots ,N_{k-d}\in\N$, for each $i$ let $\Omega_i\in\R^{m_i}$ be an aligned rectangle with integer vertices and volume $N_i$, and suppose that $\Omega\in\mc{R}_d$ is given by
\[\Omega=\Omega_1\times\cdots\times\Omega_{k-d}.\]
It is clear that every element of $\mc{R}_d$ can be written in this way, for some choice of $\{N_i\}$ and $\{\Omega_i\}$.

By Lemma \ref{lem.LCDualForm}, there is a constant $c>0$ with the property that, for each $i$, the distinct values of $\mc{L}_i(n)$, as $n$ runs over $\Omega\cap\Z^d$, are separated by a distance greater than $c/N_i$. On the other hand, by Lemma \ref{lem.LCInhomogForm}, we can choose $C>0$ so that the values of $\mc{L}_i(n)$, as $n$ runs over $C\Omega\cap\Z^d$, are at least $c/N_i$-dense. As before, this implies that $Y$ is $\mathrm{LR}_\Omega$ with respect to $\mc{R}_d$.

For the other direction of the proof, assume that (C2') does not hold. If $Y$ is $\mathrm{LR}_\Omega$ with respect to $\mc{R}_d$ then it is LR, in the usual sense. Suppose that this is the case and, for each $i$, let $m_i, \Lambda_i,$ and $X_i$ be as in the proof of Theorem \ref{thm.LRAlignedBoxes}.

We claim first of all that the same argument used in the proof of Theorem \ref{thm.LRAlignedBoxes} shows, with $\mathrm{LR}_\Omega$ instead of LR, that each $X_i$ is contained in an $m_i$-dimensional coordinate plane. To verify this, notice that the only place where the argument would differ, is in the sentence which points out that the number of integer points in a ball of volume $C|\Omega_N^{(j)}|$ is $\ll CN$. For the $\mathrm{LR}_\Omega$ argument this could be replaced by the statement that the number of integer points in $C\Omega_N^{(j)}$ is $\ll C^dN$. The rest of the proof follows exactly as before, verifying our claim.

Now by relabeling coordinates we can assume that $E$ is defined by linear forms $\{L_i\}$ as in \eqref{eqn.BlockDiagLinForms}. Since (C2') does not hold, there is an integer $i$ for which
\[\liminf_{n\rar\infty}n\|n\alpha_{i1}\|\cdots\|n\alpha_{im_i}\|=0.\]
The proof is then a consequence of Lemmas \ref{lem.LCDualForm} and \ref{lem.LCInhomogForm}, using the same argument presented at the end of the proof of Theorem \ref{thm.LRAlignedBoxes}.

For the proof of Theorem \ref{thm.LR'AlignedBoxes}, notice that in the case when $k=2d$, we must take $m_1=\cdots =m_d=1$. Then condition (C2') is precisely the condition that
\[\alpha_{i1}\in\mc{B}_1,\]
for each $i$. By Jarnik's Theorem (mentioned in Section \ref{sec.Approx}), together with a standard Hausdorff dimension argument, this set has Hausdorff dimension $d$.

\section{Canonical cut and project sets}\label{sec.Canonical}
In this section we present the proof of Theorem \ref{thm.Canonical}. We will use two lemmas to translate results about cubical cut and project sets to their canonical counterparts.
\begin{lemma}\label{lem.ChangeofWindow1}
Let $Y_1$ be a $k$ to $d$ cubical cut and project set, and let $Y_2$ be a cut and project set formed from the same data as $Y_1$, but with the canonical window. Further assume that, for each $1\le i\le d$, the point $\rho^*(e_i)$ lies on the line $\R e_j$, for some $d+1\le j\le k$. Suppose that $\mc{A}$ is a collection of bounded convex sets, with inradii uniformly bounded away from $0$. Then $Y_1$ is LR with respect to $\mc{A}$ if and only if $Y_2$ is, and $Y_1$ is $\mathrm{LR}_\Omega$ with respect to $\mc{A}$ if and only if $Y_2$ is.
\end{lemma}
\begin{proof}
We will show that there is a constant $c>0$ with the property that, for all sufficiently large $r$, the collection of all points in a ball of size $r$ in $Y_1$ uniquely determines the points in a ball of size $r-c$ in $Y_2$ and, in the other direction, that every collection of points in a ball of size $r$ in $Y_2$ uniquely determines the points in a ball of size $r-c$ in $Y_1$. This easily implies that one of the sets is LR (or $\mathrm{LR}_\Omega$) with respect to $\mc{A}$ if and only if both are.

Write $\mc{W}_2$ for the canonical window (in $F_\rho$), and let $\mc{W}'\subseteq F_\rho$ be the image under $\rho^*$ of the parallelotope generated by the standard basis vectors $e_1,\ldots ,e_d$. Then it is clear that
\[\mc{W}_2=\mc{W}_1+\mc{W}',\]
and the points in $Y_2\setminus Y_1$ correspond precisely to integer points which are mapped by $\rho^*$ into $\mc{W}_2\setminus \mc{W}_1$.

For each $1\le i\le d$, let $v_i=\pi(e_i)$, and for each subset $I\subseteq\{1,\ldots ,d\}$, let
\[v_I=\sum_{i\in I}v_i,\]
with $v_\emptyset$ taken to be $0$. For each $y\in Y_2$, let $I_y^{(1)}\subseteq\{1,\ldots ,d\}$ denote the collection of indices $i$ for which
\[y+v_i\not\in Y_1,\]
and, similarly, let $I_y^{(2)}$ denote the collection of indices $i$ for which
\[y+v_i\not\in Y_2.\]
Then, by what we said in the previous paragraph,
\[Y_2=\{y+v_I:y\in Y_1, I\subseteq I_y^{(1)}\}.\]
The reader may wish to note that that this is where we are relying on the assumption that each of the quantities $\rho^*(e_i)$ lies on a line of the form $\R e_j$. It follows that we can find a constant $c>0$ such that, for every $x\in E$ and $r>c$,
\[Y_2\cap B(x,r-c/2)=\{y+v_I:y\in Y_1\cap B(x,r), I\subseteq I_y^{(1)}\}\cap B(x,r-c/2).\]
In the other direction, we have that
\[Y_1=Y_2\setminus\{y\in Y_2:I^{(2)}_y\not=\emptyset\},\]
which means that
\[Y_1\cap B(x,r)=\{y:y\in Y_2\cap B(x,r), I^{(2)}_y=\emptyset\}.\]
We can assume that $c$ has been chosen so that $|v_i|\le c/2$ for all $i$. Therefore we have verified the assertion at the beginning of the proof, that balls of size $r$ in either one of sets, $Y_1$ or $Y_2$, uniquely determine balls of size $r-c$ in the other.
\end{proof}
Without the hypothesis that the inradii of the elements of $\mc{A}$ are uniformly bounded away from $0$, the result of this lemma would not follow immediately from the proof we have given. It is not clear to us whether or not the lemma is still valid with this assumption omitted.

A slightly less obvious fact is that the statement of this lemma is not true in general without the hypotheses on the projections of standard basis vectors (see \cite{HaynKoivWalt2015a} for  examples where the conclusion of the lemma fails). However, even in the absence of the projection hypotheses, one direction of the proof still works exactly as before, giving us the following result.
\begin{lemma}\label{lem.ChangeofWindow2}
Let $Y_1$ be a totally irrational $k$ to $d$ cut and project set, constructed with the window $\mc{W}_1$ from the previous lemma, and let $Y_2$ be a cut and project set formed from the same data as $Y_1$, but with the canonical window. If $Y_1$ is not LR (or not $\mathrm{LR}_\Omega$) with respect to $\mc{R}_d$ (or with respect to $\mc{C}_d$), then neither is $Y_2$.
\end{lemma}
The statements of Theorem \ref{thm.LRAlignedBoxes} and Corollaries \ref{cor.LRifLCTrue} and \ref{cor.LR'ifLCTrue}, with `cubical' replaced by `canonical', follow immediately from the previous two lemmas. The reader may wish to note that the proof of Theorem \ref{thm.LR'AlignedBoxes} can also be used to construct $2d$ to $d$ canonical cut and project sets which are $\mathrm{LR}_\Omega$ with respect to $\mc{R}_d$, since the corresponding canonical windows will, in that case, satisfy the hypotheses of Lemma \ref{lem.ChangeofWindow1}.

\section{An alternate choice of shapes and an open problem}\label{sec.AltChoice}
In the introduction we mentioned that certain geometric conditions must be imposed on the shapes in $\mc{A}$ in order to make the generalized definitions of LR and $\mathrm{LR}_\Omega$ interesting. Throughout the paper we have studied shapes which are subsets of the collection of convex polytopes with integer vertices. However, it would also have been natural to study collections of convex shapes with inradii uniformly bounded from below. To this end, let $\mc{C}_d'$ denote the collection of {\it convex sets} in $\R^d$ with inradii $\ge 1/2$. We may then ask whether or not, for $d>1$, there are any cubical cut and project sets which are LR (or $\mathrm{LR}_\Omega$) with respect to $\mc{C}_d'$. Note that the set $\mc{R}_d$ is a subset of $\mc{C}_d'$, and it is not difficult to show that our theorems above answer the corresponding questions about LR and $\mathrm{LR}_\Omega$ for the subset of $\mc{C}_d'$ consisting of aligned rectangles.

For $d>1$ it seems very unlikely that there are any cubical (or canonical) cut and project sets which are LR or $\mathrm{LR}_\Omega$ with respect to $\mc{C}_d'$, but we are unable to completely resolve this problem. Nevertheless, we present the following conjecture for future work.
\begin{conjecture}\label{conj.C_d'}
For $d>1$, there are no cubical or canonical cut and project sets which are $\mathrm{LR}$ or $\mathrm{LR}_\Omega$ with respect to $\mc{C}_d'$.
\end{conjecture}
The issue in applying our above arguments to try to settle this conjecture is that, in the proofs of Theorems \ref{thm.LRConvex} and \ref{thm.LR'Convex}, we used the fact that we could choose the matrix $A$ so that $\|\beta_1\|< C'/N^d$. However, $\Omega_{A,N}$ was given by
\[\Omega_{A,N}=A\cdot [-N,N]^d,\]
and the proof, in its current form, does not allow us to give a lower bound on the inradius of this shape.

As a final comment about this problem, for each $x\in\R^d$, let
\[\ell(x)=\liminf_{n\rar\infty}n\|nx_1\|\cdots \|nx_d\|.\]
In light of the above proofs, one might speculate that, in order to establish Conjecture \ref{conj.C_d'}, it might be sufficient to show that if $d>1$ then, for every $x\in\R^d$,
\[\inf_{A\in\mathrm{SL}_d(\Z)}\ell\left(Ax\right)=0.\]
This problem, which is a substantial weakening of the Littlewood conjecture, was recently resolved in an online post by Terence Tao \cite{Tao2015}. Unfortunately, the proof of Conjecture \ref{conj.C_d'} appears to require a slightly different Diophantine approximation hypothesis, which does not follow  from Tao's result. We leave it to the interested reader to carry out the details of the calculations needed to make these statements precise and, hopefully, to resolve the above conjecture.

\vspace{.2in}

{\footnotesize
\noindent Department of Mathematics, University of York,\\
Heslington, York, YO10 5DD, England\\
alan.haynes@york.ac.uk\\henna.koivusalo@york.ac.uk\\jamie.walton@york.ac.uk
}


\begin{thebibliography}{1}

\bibitem{BadzPollVela2011}
D.~Badziahin, A.~Pollington, S.~Velani:
\emph{On a problem in simultaneous Diophantine approximation: Schmidt's conjecture},
Ann. of Math. (2)  174  (2011),  no. 3, 1837-1883.


\vspace*{.1in}


\bibitem{Ball1997}
K.~Ball:
\emph{An elementary introduction to modern convex geometry},
Flavors of geometry, 1-58, Math. Sci. Res. Inst. Publ., 31, Cambridge Univ. Press, Cambridge, 1997.


\vspace*{.1in}


\bibitem{Bere2014}
V.~Beresnevich:
\emph{Badly approximable points on manifolds},
preprint, {\tt arXiv:1304.0571}.

\vspace*{.1in}

\bibitem{BereDickVela2006}
V.~Beresnevich, D.~Dickinson, S.~Velani:
\emph{Measure theoretic laws for lim sup sets},
 Mem. Amer. Math. Soc.  179 (2006),  no. 846.

\vspace*{.1in}

\bibitem{BereHaynVela2015}
V.~Beresnevich, A.~Haynes, S.~Velani:
\emph{The distribution of $n\alpha$ and multiplicative Diophantine approximation},
preprint.

\vspace*{.1in}

\bibitem{BertVuil2000}
V.~Berth\'{e}, L.~Vuillon:
\emph{Tilings and rotations on the torus: a two-dimensional generalization of  Sturmian sequences},
Discrete Math.  223  (2000),  no. 1-3, 27-53.

\vspace*{.1in}

\bibitem{BesbBoshLenz2013}
A.~Besbes, M.~Boshernitzan, D.~Lenz:
\emph{Delone sets with finite local complexity: linear repetitivity versus positivity of weights},
Discrete Comput. Geom.  49  (2013),  no. 2, 335-347.

\vspace*{.1in}

\bibitem{Cass1957}
J.~W.~S.~Cassels:
\emph{An introduction to Diophantine approximation}, Cambridge Tracts in Mathematics and Mathematical Physics, No. 45. Cambridge University Press, New York,  1957.

\vspace*{.1in}

\bibitem{CassSwin1955}
J.~W.~S.~Cassels, H.~P.~F.~Swinnerton-Dyer:
\emph{On the product of three homogeneous linear forms and the indefinite ternary quadratic forms},
Philos. Trans. Roy. Soc. London. Ser. A.  248 (1955), 73-96.


\vspace*{.1in}

\bibitem{EinsFishShap2011}
M.~Einsiedler, L.~Fishman, U.~Shapira:
\emph{Diophantine approximations on fractals},
Geom. Funct. Anal.  21  (2011),  no. 1, 14-35.

\vspace*{.1in}

\bibitem{EinsKatoLind2006}
M.~Einsiedler, A.~Katok, E.~Lindenstrauss:
\emph{Invariant measures and the set of exceptions to Littlewood's conjecture},
Ann. of Math. (2)  164  (2006),  no. 2, 513-560.

\vspace*{.1in}

\bibitem{Gall1962}
P.~Gallagher:
\emph{Metric simultaneous diophantine approximation},
J. London Math. Soc. 37 (1962) 387-390.

\vspace*{.1in}

\bibitem{GrubLekk1987}
P.~M.~Gruber, C.~G.~Lekkerkerker:
\emph{Geometry of numbers},
Second edition, North-Holland Mathematical Library, 37, North-Holland Publishing Co., Amsterdam,  1987.

\vspace*{.1in}


\bibitem{HaynKellWeis2014}
A.~Haynes, M.~Kelly, B.~Weiss:
\emph{Equivalence relations on separated nets arising from linear toral flows},
Proc. Lond. Math. Soc. (3) 109 (2014), no.5, 1203-1228.

\vspace*{.1in}


\bibitem{HaynKoivSaduWalt2015}
A.~Haynes, H.~Koivusalo, L.~Sadun, J.~Walton:
\emph{Gaps problems and frequencies of patches in cut and project sets},
preprint, {\tt arXiv:1411.0578}.

\vspace*{.1in}

\bibitem{HaynKoivWalt2015a}
A.~Haynes, H.~Koivusalo, J.~Walton:
\emph{Characterization of linearly repetitive cut and project sets},
preprint, {\tt arXiv:1503.04091}.

\vspace*{.1in}

\bibitem{Juli2010}
A.~Julien:
\emph{Complexity and cohomology for cut-and-projection tilings}, Ergodic Theory Dyn. Syst.  30  (2010),  no. 2, 489-523.

\vspace*{.1in}

\bibitem{LagaPlea2003}
J.~C.~Lagarias, P.~A.~B.~Pleasants:
\emph{Repetitive Delone sets and quasicrystals},
Ergodic Theory Dynam. Systems  23  (2003),  no. 3, 831-867.

\vspace*{.1in}

\bibitem{Lind2004}
E.~Lindenstrauss:
\emph{Adelic dynamics and arithmetic quantum unique ergodicity}
Current developments in mathematics, 2004, 111-139, Int. Press, Somerville, MA,  2006.

\vspace*{.1in}

\bibitem{Mahl1939}
K.~Mahler:
\emph{Ein \"{U}bertragungsprinzip f\"{u}r lineare Ungleichungen} (German),
\v{C}asopis P\v{e}st. Mat. Fys.  68  (1939) 85-92.

\vspace*{.1in}

\bibitem{PollVela2000}
A.~D.~Pollington, S.~L.~Velani:
\emph{On a problem in simultaneous Diophantine approximation: Littlewood's conjecture},
Acta Math. 185  (2000),  no. 2, 287-306.

\vspace*{.1in}


\bibitem{Sadu2008}
L.~Sadun:
\emph{Topology of tiling spaces},
University Lecture Series 46, American Mathematical Society, Providence, RI, 2008.

\vspace*{.1in}


\bibitem{Schm1969}
W.~M.~Schmidt:
\emph{Badly approximable systems of linear forms},
J. Number Theory 1 (1969), 139-154.

\vspace*{.1in}

\bibitem{Tao2015}
T.~Tao:
\emph{A weakening of the Littlewood conjecture},
MathOverflow, http://mathoverflow.net/q/209655 (version: 2015-06-18).


\end{thebibliography}
\end{document}